\documentclass[11pt,a4paper,headinclude,footinclude,fleqn,reqno]{amsart}                 
\usepackage[T1]{fontenc}                   
\usepackage[utf8]{inputenc}                 
\usepackage[english]{babel}       
\usepackage{graphicx}                      %
\usepackage[font=small]{quoting}            %
\usepackage{caption}  
\usepackage{amsmath}
\usepackage{amsthm}
\usepackage{amssymb}            
\usepackage[top=.8in,bottom=1.2in,left=1.5in,right=1.5in]{geometry}
\usepackage{verbatim}
\usepackage{color}
\usepackage{picture}
\usepackage{graphicx}
\usepackage{graphics}
\usepackage{comment}
\usepackage{hyperref}
\hypersetup{colorlinks,linkcolor={blue},citecolor={blue},urlcolor={red}}

\newtheorem{thm}{Theorem}[section]

\newtheorem{prop}[thm]{Proposition} 
\newtheorem{exa}{Example} 
\newtheorem{conj}[thm]{Conjecture}

\newtheorem{defn}[thm]{Definition} 
\theoremstyle{definition} 
\newtheorem{rem}[thm]{Remark} 
\theoremstyle{remark}

\def\<{\langle}
\def\>{\rangle}

{\left\{\begin{array}{@{}l@{}}}{\end{array}\right.}
\patchcmd{\abstract}{\scshape\abstractname}{\textbf{\abstractname}}{}{}
\makeatletter 
\def\@makefnmark{} 
\makeatother

\numberwithin{equation}{section}
\numberwithin{exa}{section}

\begin{document}

\title{Lagrangian surfaces with Legendrian boundary}

\author{Mingyan Li, Guofang Wang and Liangjun Weng} 
\address{\textit{School of Mathematics and Statistics, Zhengzhou University, Zhengzhou, 450001, P. R. China }}
\address{\textit{Mathematisches Institut, Albert-Ludwigs-Universit\"{a}t Freiburg, Freiburg im Breisgau, 79104, Germany}}
\email{mingyan.li@mathematik.uni-freiburg.de}

\address{\textit{Mathematisches Institut, Albert-Ludwigs-Universit\"{a}t Freiburg, Freiburg im Breisgau, 79104, Germany}}
\email{guofang.wang@math.uni-freiburg.de}

\address{\textit{School of Mathematical Sciences, University of Science and Technology of China, Hefei, 230026, P. R. China}}
\address{\textit{Mathematisches Institut, Albert-Ludwigs-Universit\"{a}t Freiburg, Freiburg im Breisgau, 79104, Germany}}

 \email{liangjun.weng@math.uni-freiburg.de}
 
\date{}
\thanks{This project is partly supported by SPP 2026 of DFG ``Geometry at infinity''}
\maketitle

\begin{abstract}In this note, we first introduce a boundary problem for Lagrangian submanifolds, analogous to the problem for free boundary hypersurfaces and capillary hypersurfaces. Then we present several interesting examples of  Lagrangian submanifolds  
	satisfying this boundary condition  and we prove a Lagrangian version of Nitsche (or Hopf) type theorem. Some problems are proposed at the end of this note.

\end{abstract}

\section{Introduction}

A hypersurface  in $\mathbb{R}^{n+1}$ with boundary lying on a (support) hypersurface $S$ is called a free boundary hypersurface  (capillary hypersurface respectively) if it intersects  the support hypersurface $S$ orthogonally
(at a constant angle $\theta\in (0,\pi)$ respectively). Such hypersurfaces play an important role in differential geometry.  A classical result of Nitsche \cite{Nit1985} states that a disk type free boundary minimal surface or capillary minimal suface 
is a totally geodesic disk. Such kind classification results hold also for constant mean curvature (CMC) surfaces.
For related results see the papers of Ros-Vergasta \cite{RV} and  Ros-Souam \cite{RosSou} for capillary minimal surfaces.

Lagrangian submanifolds play an important role in the study of closed symplectic manifolds. 
Symplectic manifolds with a convex boundary were introduced explicitly by Eliashberg-Gromov in \cite{EG}. 
In  a symplectic manifold with convex boundary, it would be interesting to consider Lagrangian submanifolds with boundary on the  convex boundary of the symplectic manifold.
The main objective of this short note is to introduce a suitable boundary problem for such Lagrangian submanifolds, similar to the free boundary problem  and the capillary boundary problem for hypersurfaces.

Let $\mathbb{C}^n=\mathbb{R}^{2n}$ be the Euclidean space equipped with standard symplectic form $\omega_0$ and  the standard complex structure $J_0$. We consider $\iota: \Sigma \to \mathbb{C}^n$ as
a Lagrangian submanifold 
with boundary $\partial \Sigma$ on the unit round sphere $\mathbb{S}^{2n-1}$. The unit ball $\mathbb{B}^{2n}$ is a convex symplectic manifold.  See \cite{EG} or Section 2 below. On  its boundary $\mathbb{S}^{2n-1}$ there is a naturally induced standard contact structure, in fact, a Sasakian structure.
We require that
the boundary $\partial \Sigma$ as a submanifold on $\mathbb{S}^{2n-1}$ is Legendrian. For the precise definition see Section 2 below. We observe that in this case the unit normal vector $\mu$ at $x\in \partial \Sigma \subset \Sigma$ lies in a plane spanned by
$x$ and $Jx$, i.e., there exists a $\theta\in [0,\pi)$ such that
$$ \mu =\sin \theta \, x + \cos \theta \, Jx.$$
We call the angle $\theta$ {\it a contact angle} and call $\Sigma$  a Lagrangian submanifold with Legendrian capillary boundary, or simply capillary Lagrangian submanfold,
if the contact angle is constant. When $\theta =\frac \pi 2$, we call $\Sigma$ a free boundary  Lagrangian submanifold, or a Lagrangian submanifold with Legendrian free boundary.

In this paper we are interested in the case $n=2$. It is trivial to see that an equatorial plane disk is a minimal Lagrangian surface with free boundary in $\mathbb{B}^4\subset\mathbb{R}^4$.
We show that the equatorial plane disks are the only such examples, even in the class of all Legendrian capillary type boundary.

\begin{thm}\label{thm in Space forms}
	\quad	Given $D:=\{(x_1,x_2): x_1^2+x_2^2\leq  1\}$.
	Let $\iota:D^2\to \mathbb{B}^4\subset \mathbb{R}^4$ be a (branched) minimal Lagrangian surface with Legendrian capillary boundary on $\mathbb{S}^3$. Then
	$\iota(D)$ is an equatorial plane disk.
\end{thm}

This result is a higher codimensional generalization of the above mentioned results of Nitsche and Ros-Souam.
We remark that in the free boundary case, Fraser-Schoen \cite{FS} obtained a much stronger result. They proved in \cite{FS} that any proper branched minimal immersion of a 2-disk into $\mathbb{B}^n$ with free boundary must be an equatorial plane disk. 
We remark that the theorem is also true for Lagrangian surfaces in a complex space form with boundary on a geodesic sphere.

Our result is optimal in the following sense. There are annulus type minimal Lagrangian surfaces with Legendrian capillary boundary, see Section 2. However, we conjecture that there is no 
annulus type minimal Lagrangian surface with Legendrian free boundary. If we remove the minimality, then there are interesting disk type Lagrangian surfaces with Legendrian capillary boundary, as a part of the Whitney sphere, which is 
the umbilical surface in the Lagrangian counterpart.

\section{Lagrangian surfaces with Legendrian capillary boundary}

\subsection{Convex symplectic manifolds}
Let $M$ be a symplectic manifold with a symplectic form $\omega$ and $S$ a closed hypersurface in $M$ enclosing $X$. One can view $X$ as a manifold with boundary $S$, with a unit outward normal vector field $N$.  A vector field $V$ on $M$ is called a Liouville vector field  if it satisfies
\begin{equation}\label{eq1}L_V \omega =\omega.\end{equation}
Here $L_V$ is the Lie derivative. $S$ is called convex  (or $\omega$-convex) if there exists a Liouville vector field $V$ in a neighborhood of $S$ such that 
\begin{equation}\label{eq2} \langle V, N \rangle >0, \quad \hbox{ on } S.
\end{equation}
See \cite{EG}.
In this case, $X$ is a symplectic manifold with a convex boundary $S$. Using this Liouville vector field, there is a naturally contact structure on $S$ induced by $X$ with the contact form
$$ \alpha:=i_V\omega.$$
In fact, by using \eqref{eq1} and \eqref{eq2} one can check 
$$(d\alpha)^{n-1}\wedge \alpha =\omega^{n-1}\wedge i_V\omega \not =0.$$
We equip $M$ with an integrable complex structure $J$, which is compatible with the symplectic form $\omega$, then a Riemannian metric $g$ on $M$ is $g(X,Y):=\omega(JX,Y)$ for all $X,Y\in TM$. 

The corresponding Reeb vector field of contact form $\alpha$ is 
\begin{equation}\label{eq0.1}
\xi:=\frac 1 a JV,
\end{equation}
where $a:=\sqrt{g(V,V)}$.

We want to use this contact structure to introduce a class of Lagrangian submanifolds in $X$ with boundary on $S$. 
We first recall the definition of Lagrangian submanifolds  in symplectic manifold $M^{2n}$ and  Legendrian submanifolds in a contact manifold. 

\begin{defn}\quad
	\emph{An immersion $\iota:\Sigma\to M$ of $n$-dimensional submanifold $\Sigma$ is called} Lagrangian submanifold \emph{if one of the following conditions holds:}
	\begin{itemize}
		\item[(1)] \emph{$\iota^*\omega=0$.}
		\item[(2)] \emph{$J(T\Sigma)=N\Sigma$, where $T\Sigma$ and $N\Sigma$ are the tangent and normal bundle of $\Sigma$.}
	\end{itemize}	
\end{defn}

\begin{defn}\quad
	\emph{$\partial\Sigma \subset S$  is called a} Legendrian submanifold \emph{in $\partial M$, if
		\begin{align}\label{eq3}	\xi\perp T\partial\Sigma
		\end{align}
		holds on $\partial\Sigma$ everywhere.}
\end{defn}

Let us focus on the 
$M:=\mathbb{C}^n$ with the standard Euclidean metric $\langle \cdot, \cdot\rangle$, the standard
symplectic form $\omega_0$ and the standard complex structure $J_0$.  
Consider a closed hypersurface $S$ (which is also called a support) with a unit outward normal vector field $N$. On $\mathbb{C}^n$ there is a well-known Liouville vector field, the position vector field $V$.
We consider the support $S$ with condition \eqref{eq1}, which is a strictly star-shaped condition. Every strictly convex hypersurface enclosing the origin in its interior satisfies this condition. Since we are interested on the support hypersurface $S$ being the unit sphere, we further focus on the Lagrangian submanifold $\Sigma$ in  $\mathbb{C}^n$ with boundary   $\partial \Sigma$ lying on the unit sphere $\mathbb{S}^{2n-1}:=\partial \mathbb{B}^{2n}$. We require that the boundary $\partial\Sigma$ as a submanifold on $\mathbb{S}^{2n-1}$ is Legendrian
i.e., \eqref{eq3} holds. In this case, $V=N$ and $a=1$. Hence $\xi =JN$.

For such a Lagrangian submanifold with boundary on $\mathbb{S}^{2n-1}$,  we have a simple, but interesting observation.

\begin{prop}\quad \label{legendrian} Let $\Sigma$ be a Lagrangian submanifold in $\mathbb{C}^n$ with   boundary  $\partial \Sigma \subset \mathbb{S}^{2n-1}$ and $\mu$ the outward unit normal vector field of $\partial \Sigma \subset \Sigma$.
	The boundary $\partial \Sigma$ is a Legendrian submanifold on $\mathbb{S}^{2n-1}$ if and only if 
	\begin{align}\label{Legendrian}
	\mu\in \mbox{span}\{N, JN\},\quad\text{on}\quad \partial\Sigma.
	\end{align}
\end{prop}
\begin{proof}\quad
	Since $\Sigma$ is  Lagrangian, along $\partial \Sigma$ we have the orthogonal decomposition,
	\begin{align}\label{eq4}
	\mathbb{C}^n=T\Sigma\oplus N\Sigma=T\Sigma\oplus JT\Sigma=\mbox{span}\{\mu,J\mu\}\oplus T\partial\Sigma\oplus JT\partial\Sigma.
	\end{align}
	
	If $\partial \Sigma$ is Legendrian, i.e.,  $\xi \perp T\partial \Sigma$, then  we have that $N\perp JT\partial\Sigma$. Combining with $N\perp  T\partial\Sigma$, it follows that 
	\begin{align*}
	N=-J\xi\in\mbox{span}\{\mu,J\mu\},
	\end{align*}
	which is equivalent to \eqref{Legendrian}. 
	
	If \eqref{Legendrian} holds, then $\xi \in \hbox {span}\, \{\mu, J\mu\}$. Due to the decomposition \eqref{eq4}, it follows that
	$\xi \perp T\partial \Sigma$, namely $\partial \Sigma$ is Legendrian.
\end{proof}

\begin{rem}\quad
	In particular, when $\mu=N$, \eqref{Legendrian} holds. Hence  from Proposition \ref{legendrian} a Lagrangian submanifold with free boundary $\partial \Sigma$ on $\mathbb{S}^{2n-1}$ has a Legendrian boundary.
\end{rem}

From now on we consider Lagrangian submanifolds with a Legendrian boundary on the support $\mathbb{S}^{2n-1}$. From \eqref{Legendrian}, there exists an angle function $\theta\in [0,\pi)$ such that
\begin{equation}\label{eq5}
\mu =\sin \theta N+ \cos \theta JN.
\end{equation}
$\theta$ is called the {\it contact angle}.
When $\theta\equiv \frac \pi 2$ on $\partial \Sigma$,  then $\Sigma$ is a free boundary Lagrangian submanifold. Therefore we introduce naturally following definition for capillary Lagrangian submanifolds.

\begin{defn}\quad \emph{Let $\Sigma$ be a Lagrangian submanifold in $\mathbb{B}^{2n}$ with a Legendrian boundary on $\partial\mathbb{B}^{2n}$. 
		We call $\Sigma$ is a}  Langrangian submanifold with a Legendrian capillary boundary \emph{(or simply a} capillary Lagrangian submanifold\emph{), if its contact angle is a constant.	
		In particular, if $\theta\equiv\frac \pi 2$,  $\Sigma$ is called a}  Langrangian submanifold with a free boundary \emph{(or simply a} free boundary Lagrangian submanifold.\emph{)} 
\end{defn}

\subsection{Interesting examples}
In this subsection we provide several interesting examples for $n=2$, which motivate us to consider such Lagrangian surfaces with boundary problem.

Consider $\mathbb{C}^2$ as the Euclidean space $\mathbb{R}^4$ with standard complex
structure $$J(x_1,x_2,y_1,y_2 ) = (-y_1,-y_2,x_1,x_2),$$ that is, a complex coordinate $x =
(x_1+ iy_1,x_2 + iy_2 )\in\mathbb{C}^2$ is identified with the real vector $x:=(x_1,x_2,y_1,y_2)\in\mathbb{R}^4.$ The standard metric and symplectic form is
\begin{align*}
g_0:=\sum_{i=1}^2 (dx_i^2+dy_i^2),\quad \omega_0:=\sum_{i=1}^{2}dx_i\wedge dy_i.
\end{align*}

\begin{exa}[\textbf{Plane  disk}]\quad
	The most trivial and simple Lagrangian surface with free boundary in the unit ball is the Lagrangian plane as
	\begin{align*}
	\Sigma:=\{(x_1,x_2,0,0)\}\subset\mathbb{C}^2.
	\end{align*}
	
\end{exa}
It is easy to see from above that the intersection of plane $\Sigma$ with the unit ball is a Lagrangian minimal surface with a free boundary on $\mathbb{S}^3.$	

\begin{exa}\quad
	The \textbf{Lagrangian catenoid} can be written as
	\begin{align*}
	\mathbb{LC}^n:=\{(x,y)\in\mathbb{R}^{2n}: |x|y=|y|x, \mbox{Im}(|x|+i|y|)^n=1, |y|<|x|\tan\frac{\pi}{n}\}.
	\end{align*}
	In particular, for $n=2$, it can be written as a holomorphic surface like
	\begin{align*}
	\mathbb{LC}^2&:=\{(z,\frac{1}{z})\in\mathbb{C}^2: z\neq 0\}.
	\end{align*}
	Or in parametrization $(t,\theta)$ form
	\begin{align*}
	\mathbb{LC}^2:=\{(t\cos\theta,t\sin\theta,\frac{1}{t}\cos\theta,\frac{1}{t}\sin\theta)\in\mathbb{C}^2| t>0, 0\leq \theta<2\pi\}.
	\end{align*}\end{exa}

It is well-known that $\mathbb{LC}^n$ is a topological  $\mathbb{R}\times\mathbb{S}^{n-1}$, its induced metric is conformally flat, but not cylindrical. In \cite{CU1999}, 
the authors proved that if a Lagrangian minimal submanifold in $\mathbb{C}^n$ is foliated by round $(n-1)$-spheres, it is congruent to a Lagrangian catenoid.
Hence it is natural Lagrangian counterpart of catenoid.
\begin{prop}\quad For any $r_0>\sqrt 2$, 
	$(\frac 1 {r_0} \mathbb{LC}^2) \cap \mathbb{B}^4$ is a Lagrangian, minimal annulus with a Legendrian  capillary boundary.
	Moreover, the two components of its boundary are great circles.\end{prop}

\begin{proof}\quad
	Consider the ball with radius $r_0>\sqrt 2$ in $\mathbb{R}^4$ 
	\begin{align*}
	\mathbb{	B}_{r_0}^4:=\{x\in\mathbb{R}^4: |x|=r_0\}.
	\end{align*} 
	Define  $\Sigma:=  \mathbb{LC}^2\cap \mathbb{B}^4_{r_0}$, and set $$T^2_\pm := \frac{r_0^2}{2}\pm \sqrt {\frac {r_0^4}4-1}.$$
	Then 
	$$\Sigma=\{X=(t\cos\theta,t\sin\theta,\frac{1}{t}\cos\theta,\frac{1}{t}\sin\theta)\in\mathbb{C}^2| t\in [T_-, T_+], 0\leq \theta<2\pi\}.$$
	The tangent space of $T\Sigma$ is spanned by
	\begin{align*}
	\begin{cases}
	X_t=(\cos\theta,\sin\theta,-\frac{\cos\theta}{t^2}, -\frac{\sin\theta}{t^2}),\\
	X_{\theta}=(-t\sin\theta,t\cos\theta,-\frac{\sin\theta}{t}, \frac{\cos\theta}{t}).
	\end{cases}
	\end{align*}
	One can check that  the normal space of $N\Sigma$ is spanned by
	\begin{align*}
	\begin{cases}
	\nu_1 = J X_t /|X_t|=  \frac 1 {\sqrt{ 1+t^4}}(\cos \theta, \sin\theta, t^2  \cos \theta,  t^2  \sin \theta ), \\
	\nu_2 = J X_\theta /|X_\theta|= \frac 1 {\sqrt{ 1+t^4}} (\sin \theta ,-\cos\theta , -t^2 \sin\theta, t^2 \cos\theta).
	\end{cases}	\end{align*}
	Hence this surface is Lagrangian.  One can also see that both two boundary components 
	$$\partial_\pm \Sigma =\{(t\cos\theta,t\sin\theta,\frac{1}{t}\cos\theta,\frac{1}{t}\sin\theta)\in\mathbb{C}^2| t=T_\pm, 0\leq \theta<2\pi\}.
	$$
	And its boundary is Legendrian, because of
	\[
	\langle {X_\theta}, JX \rangle=0,\quad \text{on}\quad \partial _\pm\Sigma.
	\]
	One  can check that the boundary is  capillary. In fact, the conormal $\mu$ is perpendicular to the space spanned by $\nu_1,\nu_2$ and $X_{\theta}$, and hence we have that $\mu$ is parallel to $X_t$ on $\partial_\pm \Sigma$.
	Therefore
	\begin{align*}
	\mu=\mp\frac{(-T_\pm^2\cos\theta,-T_\pm^2 \sin\theta,\cos\theta,\sin\theta)}{\sqrt{1+T_\pm^4}}.
	\end{align*}
	Then
	\begin{align*}
	\langle \mu,\overline{N}\rangle&:=\langle \mu,\frac{X}{|X|}\rangle
	\\&=\frac{\mp 1}{r_0\sqrt{1+T_\pm^4}}\big\langle (-T_\pm^2\cos\theta,-T_\pm^2 \sin\theta,\cos\theta, \sin\theta),(T_\pm \cos\theta,T_\pm\sin\theta,\\&\quad \frac{1}{T_\pm}\cos\theta,\frac{1}{T_\pm}\sin\theta)\big\rangle
	\\&=\frac{\sqrt{r_0^4-4}}{r_0^2}\in (0,1).
	\end{align*}
	
\end{proof}

\begin{exa}\quad
	The \textbf{Whitney sphere} is the immersion $f_{r,c}$ given by
	\begin{align}&f_{r,c}:\mathbb{S}^2\to \mathbb{R}^{4}\\&
	(x_1,x_2,x_{3})\mapsto \frac{r}{1+x_{3}^2}(x_1,x_2,x_1x_{3}, x_2x_{3})+c,
	\end{align}
	where $r$ is a positive constant and
	$c$ is a point in $\mathbb{C}^2$. We will refer to the quantities $r$ and $c$ as the radius and the center of the Whitney sphere.\end{exa}

Up to dilations and
translations, all Whitney spheres can be identified with $f_{1,0}$.
These Whitney spheres  play the role of umbilical hyperspheres in the Euclidean space $\mathbb{R}^{n+1}$ of the Lagrangian setting.

From the discussion in \cite{RU}, we know $JH$ is a conformal vector field for the Whitney sphere, where $H$ is the mean curvature vector field of the corresponding Lagrangian submanifold. Equivalently, $i_{JH}\omega$ is the conformal Maslov form, where $i_{JH}\omega(\cdot):=\omega(\cdot,H)$ is called Maslov form.

\begin{prop}\quad For any $r>1$, $f_{r,0} (\mathbb{S}^2)\cap \overline{\mathbb{B}}^4$ is a Lagrangian surface with a Legendrian capillary boundary, which has a  conformal Maslov form.
	Moreover, it consists of two components, each of them is a topological disk.
\end{prop}
\begin{proof}\quad It is clear that for any $r>1$, $f_{r,0} (\mathbb{S}^2)\cap \overline{\mathbb{B}}^4$ is nonempty. One can check the Proposition directly as above. Here we use the above example to show the Proposition.
	It is well-known that the Whitney sphere can be viewed as an inversion of the holomorphic curve
	$$(z, \frac 1 z),$$
	which, in turn, can be viewed as the Lagrangian catenoid. Let us consider the intersection of the above holomorphic curve
	with the ball $\mathbb{B}_{r_0}$ with $r_0>\sqrt 2$. After taking an inversion with respect to the unit sphere, we get a Whitney sphere intersecting with the ball $\mathbb{B}_{\frac 1 {r_0}}$.
	Since the inversion preserves the intersecting angle, we obtain this Whitney sphere intersecting with $\partial \mathbb{B}_{\frac 1{r_0}}$ at a constant angle. Moreover the Reeb vector field of $\partial \mathbb{B}_{r_0}$ is mapped to a  multiple of the
	Reeb vector field of  $\partial \mathbb{B}_{\frac 1{r_0}}$. In view of the definition of Legendrian submanifolds, \eqref{eq3}, it is clear that this inversion maps a Legendrian curve on $\partial \mathbb{B}_{{r_0}}$ into a Legendrian curve in $\partial \mathbb{B}_{\frac 1{r_0}}$. 
	The statement follows after a rescaling.
\end{proof}

\subsection{A Lagrangian Joachimsthal property}

Let $\Sigma \subset \mathbb{R}^3$ be a surface in $\mathbb{R}^3$. Since the second fundamental form  $h$ is symmetric, at any fixed point $p$ on $\Sigma$ one can find a suitable local coordinate around this point such that $h$ is diagonal at
$p$, i.e., $h_{12}( p)=0$. If this property is true for a curve $\gamma$, to be more precise, for any point at the curve, there exists a local coordinate such that $\frac \partial {\partial x_1}$ is the tangential vector field  of this  curve and 
$h_{12}=0$ along the curve,
then this curve $\gamma$ is called a {\it curvature line}.  Note that for hypersurfaces one can also introduce the notion of curvature lines.
Following is  a classical result of Joachimsthal: 

\

{\it Let $M_i\subset \mathbb{R}^3$  $(i=1,2)$ be  two regular and oriented surfaces
	such that $M_1$ intersects  $M_2$ along a regular $\gamma$ with a constant angle.
	Then  $\gamma$ is a principal curvature
	line of $M_1$  if and only if it is a curvature line of $M_2$.}

\

As a corollary of this result, one can see that for either a free boundary surface or capillary surface in a ball, its boundary is a curvature line.
In fact, it is not difficult to check that if $\Sigma \subset \mathbb{R}^3$ intersects with $\mathbb{S}^2$ along its boundary $\partial \Sigma$, then the angle is constant if and only if $\partial \Sigma$ is a curvature line.

\

Now  we introduce analogously  the notion of curvature line for Lagrangian surfaces as higher codimensional submanifolds as follows. Let $\Sigma$ be a Lagrangian surface in $
\mathbb{R}^4$ and  $h$ be  its second fundamental form. We view it as  a 3-tensor by using
$$ A(e_i, e_j, e_k)=\langle h(e_i,e_j), Je_k \rangle,$$
for any orthonormal frame $\{e_i\}_{i=1}^2$ on $\Sigma$. It is well-known that $A$ is a symmetric 3-tensor (or see below). For such a 3-tensor, at any fixed point $p\in \Sigma$ one can find a local coordinate around $p$ such that $A(e_1,e_2,e_2)(p)=0$, where $e_1$ is the tangent vector field and $e_2$ is conormal vector field to the curve. We call a curve $\gamma \subset \Sigma$ is a line of curvature (or curvature line), if this property holds on the whole curve.

Now we can show 

\begin{prop}\label{J}\quad  Let $\Sigma\subset  \mathbb{R}^4$ be a Lagrangian surface with Legendrian boundary $\partial \Sigma$ on unit sphere $\mathbb{S}^3$ and contact angle $\theta$. Then $\partial \Sigma$ is a curvature line if and only if $\theta$ is a constant, i.e.,
	$\Sigma$ is a Lagrangian surface with Legendrian capillary boundary.
\end{prop}
\begin{proof}\quad
	Let $\mu$ be the unit outward conormal  vector field of $\partial\Sigma\subset \Sigma$ and $T$ be the tangential vector field of the curve $\partial \Sigma$.
	By Proposition \ref{legendrian} we have 
	\begin{equation*}\mu =\sin \theta N+ \cos \theta JN,
	\end{equation*}
	where $\theta$ is the contact angle. Next we compute
	\begin{eqnarray*}
		A (T,\mu, \mu) =\langle h(T,\mu),J \mu \rangle.
	\end{eqnarray*}
	
	Since $\mathbb{S}^3$ is umbilic, denote $D$ as the gradient in $\mathbb{R}^4$, we have $D_T N=T$. It follows that
	$D_T \mu =D_T\theta\cdot (\cos \theta N- \sin \theta JN)  + \sin\theta T +\cos \theta JT$. This formula, together with the Legendrian property $\langle T, JN \rangle =0$ implies
	\begin{eqnarray*}
		A (T, \mu, \mu) =-D_T\theta ,
	\end{eqnarray*}
	i.e., $A(T,\mu, \mu)$ vanishes if and only if  $\theta$ is constant, which finishes our proof.
\end{proof}

This Proposition can be viewed as a Lagrangian counterpart of the classical Joachimsthal Theorem.
\subsection{Lagrangian submanifolds with conformal Maslov form}
Consider $\iota :\Sigma\to\mathbb{C}^n$ be Lagrangian immersion of an $n$-dimensional submanifold $\Sigma$, we choose an orthonormal basis $\{e_i\}_{i=1}^n$ of $T\Sigma$ with respect to the induced metric $g$ on $\Sigma$, and $\{Je_i\}_{i=1}^n$ is the orthonormal basis for $N\Sigma$. The coefficient of the second fundamental form is given by $A_{ijk}:=\langle h(e_i,e_j), Je_k\rangle$ and $H:=\sum\limits_{i=1}^{n}h(e_i,e_i)$ is the mean curvature vector field. We still use the usual conventions of
summation for these indices. Note that we mark the normal frame with
an upper index and the tangential frame with a lower index. 

We introduce a  $(0,3)$-tensor $\breve{A}\in \Gamma \Sigma\big(T\Sigma\otimes T\Sigma\otimes T\Sigma\big)$  with its coefficient being
\begin{equation*}
\begin{split}
\breve{A}_{ijk}&:=\breve{A}(e_i,e_j,e_k)\\&=\langle h(e_i,e_j),Je_k\rangle-\frac{1}{n+2}\big[ \langle e_i,e_j\rangle \langle H,Je_k\rangle+\langle Je_i,H\rangle \langle Je_j,Je_k\rangle
\\&\quad +\langle Je_j,H\rangle \langle Je_i,Je_k\rangle\big]
\\&=A_{ijk}-\frac{1}{n+2}\big(H^kg_{ij}+H^ig_{jk}+H^jg_{ik}\big),
\end{split}
\end{equation*}
for any tangent vector field $e_i\in T\Sigma$.
This tensor will play a fundamental role in our following study. 

Notice that from Gauss-Weingarten equations, it is known that
\begin{align*}
\nabla_{e_i}e_j+h(e_i,e_j)=D _{e_i}e_j,\quad D_{e_i}Je_j= -W_j(e_i)+\nabla^{\perp}_{e_i}(Je_j),\end{align*}
where $\nabla$ and $D$ are the gradients in $\Sigma$ and $\mathbb{C}^n$ respectively,  and $\nabla^{\perp}$ denotes the connection in the normal bundle of $\Sigma$.  In particular, for Lagrangian submanifold, it holds 
\begin{align*}
\begin{cases}
J(\nabla_{e_i}e_j)=\nabla_{e_i}^{\perp }Je_i,
\\
Jh(e_i,e_j)=-W_j(e_i),
\end{cases}
\end{align*}
for any $e_i\in T\Sigma$. $W_j$ denotes the Weingarten map corresponding to the normal vector field $Je_j$.
Hence 
\begin{align*}
\langle h(e_i,e_j),Je_k\rangle &=\langle JW_j(e_i),Je_k\rangle=\langle W_j(e_i),e_k\rangle=\langle -D_{e_i}Je_j,e_k\rangle
\\&=\langle Je_j,h(e_i,e_k)\rangle,
\end{align*}
that is, $\langle h(\cdot,\cdot),J\cdot\rangle$ defines a totally symmetric trilinear form on $T\Sigma$. Then it follows that $\breve{A}$ is also a totally symmetric trilinear form on $T\Sigma$, and is trace free with any $2$-symbols. In particular, Castro-Urbano \cite{CU1993} and Ros-Urbano \cite{RU} classified all Lagrangian submanifolds with vanishing $\breve A$. They are a part of Lagrangian subspace or Whitney sphere. 

\begin{prop}\label{codazzi lag}\quad If $\iota:\Sigma\to \mathbb{C}^n$ is a Lagrangian immersion, then we have 
	\begin{equation} 
	{(n+2)\nabla^* \breve{A}=\mbox{\rm div}(JH) g-n\nabla (JH),}
	\end{equation}
	where $\nabla^*$ denotes the formal adjoint of the operator $\nabla$, i.e. $\nabla^* Q:=
	-\nabla_{e_i} Q (e_i,\ldots)$ for any $k$-form $Q$ along $\iota(\Sigma)$.
\end{prop}

\begin{proof}\quad
	From the Codazzi equation, we know that $$A_{ijk,l}=A_{ilk,j}.$$
	By taking contraction, it follows that
	\begin{equation}\label{Codazzi eq}
	\nabla^{\perp}H=-\nabla^* A.
	\end{equation}
	A direct computation and combining with equation \eqref{Codazzi eq}, we deduce that
	\begin{equation*}
	\begin{split}
	(\nabla^* \breve{A})_{kij,k}&=-\nabla_{e_k}\breve{A}(e_k,e_i,e_j)=-
	\breve{A}_{kij,k}\\&=-\big[A_{kij,k}-\frac{1}{n+2}(H^{j}_{,k}g_{ki}+H^k_{,k}g_{ij}+H^i_{,k}g_{jk})\big],
	\end{split}
	\end{equation*}
	which implies
	\begin{equation*}
	\begin{split}
	\sum_{k}(\nabla^* \breve{A})_{kij,k}=\frac{1}{n+2}\sum_{k}H^k_{,k}g_{ij}-\frac{n}{n+2}H^{j}_{,i}.
	\end{split}
	\end{equation*}
	That is,
	\begin{equation*} (n+2)\nabla^* \breve{A}=\mbox{div}JH g -n\nabla (J H),
	\end{equation*}
	hence we complete the proof.
\end{proof}

We recall 
\begin{defn} \quad \emph{A Lagrangian immersion $\iota:\Sigma\to\mathbb{C}^n$ is called a} Lagrangian submanifold with conformal Maslov form \emph{if $\mbox{div}JH g -n\nabla (J H)=0$ or equivalently, $\nabla^*\breve{A}=0$  holds everywhere.} 
	
\end{defn}

It is easy to observe that if a Lagrangian immersion is a minimal submanifold or has parallel mean curvature vector field, then it is a Lagrangian submanifold with conformal Maslov form.
In \cite{CU1993}, Castro-Urbano introduced an interesting 3-form, which is  a Hopf type differential form and proved

\begin{prop}[\cite{CU1993}]\label{holo criterion}\quad Given
	$\iota:D \to \mathbb{C}^2$ is a Lagrangian immersion, then it is a Lagrangian with conformal Maslov form if and only if the associated cubic differential form $\Phi(z):=A(\partial_z,\partial_z,\partial_z)dz^3$ is holomorphic, where $A(\partial_z,\partial_z,\partial_z):=\langle h(\partial_z,\partial_z),J\partial_z\rangle$, and $\langle\cdot,\cdot\rangle, h,J$ are extended $\mathbb{C}$-linearly to the complexified bundles.
\end{prop}
\begin{proof}\quad For the completeness we give the proof here.
	Let $z:=x_1+ix_2$ be the standard complex coordinate on disk $D$ and also isothermal coordinate with respect to the standard Euclidean metric on $\iota(D)$, and 
	\begin{equation*}
	\begin{cases}
	\frac{\partial}{\partial z}:=\frac{1}{2}\big(\frac{\partial}{\partial x_1}-\sqrt{-1}\frac{\partial}{\partial x_2}\big),\\
	\frac{\partial}{\partial \overline{z}}:=\frac{1}{2}\big(\frac{\partial}{\partial x_1}+\sqrt{-1}\frac{\partial}{\partial x_2}\big).
	\end{cases}
	\end{equation*}
	Then the induced metric is
	\begin{align*}
	g=2|\iota_z|^2|dz|^2:=e^{2\lambda} |dz|^2,
	\end{align*}
	which gives
	\begin{align*}
	\nabla_{\partial_z}\partial_{\overline{z}}=0,\quad \nabla_{\partial_z}\partial_{{z}}=   2{\lambda_z} \partial_z,
	\end{align*}
	here $\nabla$ is the Levi-Civita connection of the first fundamental form $g$ on $\Sigma:=\iota(D)$.
	
	If $\Sigma$  is a Lagrangian with conformal Maslov form, i.e. $\nabla^*\breve{A}=0$, using Proposition \ref{codazzi lag} for $n=2$, it follows that
	\begin{align*}
	\mbox{div} JH g=2\nabla (JH).
	\end{align*}
	Combining with Codazzi equation, it is equivalent to
	\begin{align}\label{codazzi lag for n=2}
	\begin{cases}
	A_{111,1}=A_{222,2},\\
	A_{111,2}+A_{122,2}=0,
	\end{cases}
	\end{align}
	where $A_{***,*}$ are the coefficients of the covariant derivative of tensor $A$ with respect to $\partial_{x_1},\partial_{x_2}$, for instance $A_{111,1}:=(\nabla_{\partial_{x_1}} A)(\partial_{x_1},\partial_{x_1},\partial_{x_1})$, etc. Due to $\nabla_{\partial_z}\partial_{\overline{z}}=0$, we obtain 
	\begin{align*}
	\partial_{\overline{z}}\big[A(\partial_z,\partial_z,\partial_z)\big]&=
	(	\nabla_{\partial_{\overline{z}}} A)(\partial_z,\partial_z,\partial_z )+3A(\nabla_{\partial_{\overline{z}}}\partial_z,\partial_z,\partial_z)
	\\&=\frac{1}{16}\Big[(\nabla_{x_1+ix_2}A) (\partial_{x_1},\partial_{x_1},\partial_{x_1} )-3(\nabla_{x_1+ix_2}A)(\partial_{x_1},\partial_{x_2},\partial_{x_2})
	\\&\quad -3i(\nabla_{x_1+ix_2}A)(\partial_{x_1},\partial_{x_1},\partial_{x_2})+i(\nabla_{x_1+ix_2}A)(\partial_{x_2},\partial_{x_2},\partial_{x_2}) \Big]
	\\&=\frac{1}{16}\big(A_{111,1}-A_{222,2}-2iA_{112,1}-2iA_{122,2}\big)
	\\&=0,
	\end{align*}
	where we have used the Codazzi equation in the second equality and equation \eqref{codazzi lag for n=2} on the last equality above.

\end{proof}

\subsection{Proof of main theorem}
After the above preparation, we can now  give the proof of the main theorem. 
\begin{proof}[\textbf{Proof of Theorem \ref{thm in Space forms}}] 	\quad
	From  the above discussion, we know that a minimal Lagrangian submanifold is of Lagrangian with conformal Maslov form. From Proposition \ref{holo criterion} we know that the cubic differential form 
	\begin{align*}
	\Phi(z):=A(\partial_z,\partial_z,\partial_z)dz^3,
	\end{align*}
	is holomorphic in $D$. Now we choose the polar coordinate in disk with $z=re^{i\theta}$, then it follows that
	\begin{align*}
	\begin{cases}
	\frac{\partial}{\partial z}=\frac{e^{-i\theta}}{2}\big(\partial_r-\frac{i}{r}\partial_{\theta}\big),\\
	\frac{\partial}{\partial \overline{z}}=\frac{e^{i\theta}}{2}\big(\partial_r+\frac{i}{r}\partial_{\theta}\big).\end{cases}
	\end{align*}Hence, we have
	\begin{align*}
	8z^3A(\partial_z,\partial_z,\partial_z)={r^3}\big(A_{rrr}-\frac{3}{r^2}A_{r\theta\theta}-\frac{3i}{r}A_{rr\theta}+\frac{i}{r^3}A_{\theta\theta\theta}\big),
	\end{align*}
	on $\partial D$. Now we claim that its imaginary part  vanishes on $\partial D$. In fact, due to $\frac \partial {\partial r}$ is the normal vector on the boundary and $\frac \partial {\partial \theta}$ is the tangential vector,
	from Proposition \ref{J}, we have $A_{rr\theta}=0$ on the boundary $\partial D$. The minimality implies that 
	$A_{\theta\theta\theta}=-A_{rr\theta}=0$. 
	And hence the imaginary part  vanishes on the boundary. It follows that    $z^3A(\partial_z,\partial_z,\partial_z)\equiv c$. Since it vanishes at the origin, 
	$z^3A(\partial_z,\partial_z,\partial_z)\equiv 0$. 
	It follows, together with the minimality, that
	the second fundamental form vanishes over $D$. Then  $\Sigma=\iota(D)$ must be totally geodesic disk.
\end{proof}
For minimal Lagrangian surfaces with boundary  and Legendrian surfaces in $\mathbb{S}^5$ we refer to Chen-Yuan \cite{CY} and Fu \cite{F}.

\subsection{Problems} To end this note, we propose two conjectures.

\begin{conj}\label{conj1}\quad Any immersed Lagrangian disk in $\mathbb{R}^4$ with conformal Maslov form and Legendrian capillary boundary on $\mathbb{S}^3$, then it must be either totally geodesic or one of the Whitney spheres given in Example 2.9.
\end{conj}

It was proved in \cite{CU1993} that any Lagrangian sphere with conformal Moslov form is a Whitney sphere. For the problem with boundary we have the following difficulty.
When the surface in Theorem 1.1 is not  minimal, we have only that the imaginary part  of the holomorphic form $\Phi$
is
$$A_{\theta\theta\theta}-3A_{rr\theta}=A_{\theta\theta\theta},$$
where $A_{rr\theta}=0$ on $\partial D$ is used by Proposition \ref{J} for Legendrian capillary boundary. It is not easy to show that $A_{\theta\theta\theta}=0$ on $\partial\Sigma$, without extra conditions.
If assume further that any the boundary is a circle on $\mathbb{S}^3$, then one can see that $A_{\theta\theta\theta}=0$ on $\partial\Sigma$. And hence, in this case the conjecture is true.

As mentioned in the Introduction we conjecture that there is no annulus type minimal Lagrangian surface with Legendrian free boundary.  In fact, we conjecture
\begin{conj}\label{conj2}\quad Any embedded annulus type minimal Lagrangian surface with Legendrian capillary boundary on $\mathbb{S}^3$  is one of  the examples given in Example 2.7.
\end{conj}

This conjecture is  a Lagrangian  counterpart of the conjecture of Fraser-Li in \cite{FL}: {\it Any embedded minimal annulus with free boundary on $\mathbb{S}^2$ is the critical catenoid.}
In the class of Lagrangian catenoid discussed in Example 2.7 there is no example with Legendrian free boundary.

\textbf{Acknowledgements:}{ This work is partly supported by SPP 2026 of DFG ``Geometry at infinity''. A part of this work was carried out when the second named author visited UBC. He would like to thank Jingyi Chen and the department for their warm hospitality.}

\end{document}